\documentclass[11pt]{article}
\usepackage{graphicx}
\usepackage{amsfonts}
\usepackage{texdraw}
\usepackage{amssymb, amsmath, amsthm}
\usepackage{epsfig}
\usepackage[english]{babel}
\usepackage{latexsym}
\usepackage{amscd}

\newtheorem{proposition}{Proposition}

\def\neweq#1{\begin{equation}\label{#1}}
\def\endeq{\end{equation}}
\def\eq#1{(\ref{#1})}

\newcommand{\R}{\mathbb{R}}

\begin{document}

\title{The role of aerodynamic forces\\
in a mathematical model for suspension bridges}

\author{Elvise BERCHIO $^\sharp$ - Filippo GAZZOLA $^\dagger$}
\date{}
\maketitle

\begin{center}
{\small $^\sharp$ Dipartimento di Scienze Matematiche-Politecnico di Torino-Corso Duca degli Abruzzi 24-10129 Torino, Italy\\
        $^\dagger$ Dipartimento di Matematica - Politecnico di Milano - Piazza Leonardo da Vinci 32 - 20133 Milano, Italy\\
{\tt elvise.berchio@polito.it, filippo.gazzola@polimi.it}}
\end{center}

\begin{abstract}
In a fish-bone model for suspension bridges previously studied by us in \cite{bergaz} we introduce linear aerodynamic forces. We numerically analyze the role
of these forces and we theoretically show that they do not influence the onset of torsional oscillations. This suggests a new explanation for the origin
of instability in suspension bridges: it is a combined interaction between structural nonlinearity and aerodynamics and it follows a precise pattern.
This gives an answer to a long-standing question about the origin of torsional instability in suspension bridges.\par\noindent
\end{abstract}

\section{Introduction}
Since the Federal Report \cite{ammann}, it is known that the crucial event causing the collapse of the Tacoma Narrows Bridge was a sudden change from a vertical
to a torsional mode of oscillation. Several studies were done on this topic, see \cite{larsen,pittel,scantom} but a full explanation of the origin of
torsional oscillations is nowadays still missing; see also the updated monograph \cite{scott} and references therein. In two recent papers the onset
of torsional oscillations was attributed to a structural instability.
In \cite{argaz} a model of suspension bridge composed by several coupled (second order) nonlinear oscillators has been proposed. By using
suitable Poincar\'e maps, it has been proved that when enough energy is present within the structure a resonance may occur, leading to an energy
transfer between oscillators, from vertical to torsional. The results in \cite{argaz} are purely numerical. We found a similar answer in \cite{bergaz}
by analyzing a different mathematical model, named {\em fish-bone}. In this model, the main span of the bridge, which has a rectangular shape with two long
edges and two shorter edges, is seen as a degenerate plate fixed and hinged between the towers. The midline of the roadway is seen as a beam
with cross sections that are seen as rods free to rotate around their barycenters located on the beam. The degrees of freedom are the vertical
displacement of the beam $y$, which is positive in the downwards direction, and the angle $\theta$ of rotation of the cross sections with respect to
the horizontal position. The roadway is assumed to have length $L$ and width $2\ell$ with $2\ell\ll L$. By considering the kinetic energy of a rotating
object and the bending energy of a beam, the following system is obtained in \cite{bergaz}:
\begin{equation}\label{system0}
\left\{\begin{array}{ll}
\!\!My_{tt}\!+\!EI y_{xxxx}\!+\!f(y\!+\!\ell\sin \theta)\!+\!f(y\!-\!\ell\sin \theta)\!=\!0\quad & 0<x<L,\ t>0,\\
\!\!\frac{M\ell}{3}\theta_{tt}\!-\!\mu\ell\theta_{xx}\!+\!\cos\theta(f(y\!+\!\ell\sin \theta)\!-\!f(y\!-\!\ell\sin \theta))\!=\!0\quad & 0<x<L,\ t>0,
\end{array}
\right.
\end{equation}
where $M$ is the mass of the rod,  $\mu>0$ is a constant depending on the shear modulus and the moment of inertia of the pure torsion, $EI>0$ is the
flexural rigidity of the beam, $f$ includes the restoring action of the prestressed hangers and the action of gravity.
To \eq{system0} we associate the boundary-initial conditions
\begin{equation}\label{boundaryc}
y(0,t)\!=\!y_{xx}(0,t)\!=\!y(L,t)\!=\!y_{xx}(L,t)\!=\!\theta(0,t)\!=\!\theta(L,t)\!=\!0\qquad t\geq 0,
\end{equation}
\begin{equation}\label{initial0}
y(x,0)\!=\!\eta_0(x),\,y_t(x,0)\!=\!\eta_1(x),\,\theta(x,0)\!=\!\theta_0(x),\,\theta_t(x,0)\!=\!\theta_1(x)\quad 0<x<L.
\end{equation}

For a linear force $f$ the two equations in \eq{system0} decouple: this case was studied in \cite{moore}. In the nonlinear case, well-posedness of the
problem was shown in \cite{holubova}. For a suitable nonlinear $f$, in \cite{bergaz} we gave a detailed explanation of how internal resonances occur
in \eq{system0}, yielding instability. The aim of this analysis was purely qualitative and the bridge was seen as an isolated system with no dissipation
and no interactions with the surrounding air. In particular, both theoretical and numerical results were given proving that the onset of large
torsional oscillations is due to a resonance which generates an energy transfer between different oscillation modes. More precisely, when the bridge
is oscillating vertically with sufficiently large amplitude, part of the energy is suddenly transferred to a torsional mode giving rise to wide
torsional oscillations. Estimates of the energy threshold for stability were obtained both theoretically and numerically, see Section \ref{One}.\par
Our purpose in \cite{bergaz} was to emphasize the structural behavior of the bridge without inserting any interaction with the surrounding air.
This procedure was also followed by Irvine \cite[p.176]{irvine} who comments his own approach by writing:
\begin{center}
\begin{minipage}{6.2in}
{\em In this formulation any damping of structural and aerodynamic origin has been ignored... We could include aerodynamic damping which is perhaps the most
important of the omitted terms. However, this refinement, although frequently of significance, yields a messy flutter determinant that requires a numerical
solution.}
\end{minipage}
\end{center}
This comment says two things. First, that it was a good starting point to study \eq{system0} as an isolated system. Second, that, in order to have more
accurate responses, the subsequent step should be to insert aerodynamic forces in the model. This refinement of the model \eq{system0} was also suggested to us by Paolo Mantegazza, a distinguished aerospace engineer at the Politecnico of Milan, and motivates the present paper. In order to better highlight the role of the aerodynamic forces, we do not insert in the model any other external action. We will show, both numerically and theoretically, that the threshold of instability of the system is independent of aerodynamic forces. This suggests a new pattern for the aerodynamic and structural mechanisms which give rise to oscillations in suspension bridges, see Section \ref{conclusion}.

\section{One mode approximation of the fish-bone model}\label{One}
We first introduce some simplifications of the model which, however, maintain its original essence and its main structural features. First, up to scaling
we may assume that $L=\pi$ and $M=1$. Then, since we are willing to describe how small torsional oscillations may suddenly become larger ones, we use
the following approximations: $\cos\theta \cong 1$ and $\sin \theta \cong \theta$; see \cite{bergaz} for a rigorous justification of this choice.
Since our purpose is merely to describe the qualitative phenomenon, we may take $EI\left(\frac{\pi}{L} \right)^{4}=3\mu\left(\frac{\pi}{L} \right)^{2}=1$ although these parameters may be fairly different in actual bridges. For the same reason, the choice of the nonlinearity is not of fundamental importance;
it is shown in \cite{argaz} that several different nonlinearities yield the same qualitative behavior for the solutions. Whence, as suggested by
Plaut-Davis \cite[Section 3.5]{plautdavis}, we take $f(s)=s+ s^3$. Finally, we set $z:=\ell \theta$ and the system \eq{system0} becomes
\begin{equation}\label{systemgamma}
\left \{ \begin{array}{ll}
y_{tt}+y_{xxxx}+2y(1+ y^2+3 z^2)=0\qquad & 0<x<\pi,\ t>0,\\
z_{tt}-z_{xx}+6z(1+3 y^2+ z^2)=0\qquad & 0<x<\pi,\ t>0\, .
\end{array}\right.
\end{equation}
To \eq{systemgamma} we associate some initial conditions which determine the conserved energy $E$ of the system, that is,
\begin{eqnarray*}
E &=& \frac{\|y_t(t)\|_2^2}{2}+\frac{\|z_t(t)\|_2^2}{6}+\frac{\|y_{xx}(t)\|_2^2}{2}+\frac{\|z_x(t)\|_2^2}{6}\\
\ & \ & +\int_0^\pi \Big(y(x,t)^2+z(x,t)^2+3 z(x,t)^2y(x,t)^2+\frac{y(x,t)^4}{2}+\frac{z(x,t)^4}{2} \Big)\, dx\, .
\end{eqnarray*}
Existence and uniqueness of solutions were proved in \cite{holubova} by performing a suitable Galerkin procedure, see also \cite{bergaz} where
more regularity was obtained. The proof is constructive: to \eq{systemgamma} we associate the functions
\neweq{fouriern}
y^m(x,t)=\sum_{j=1}^m y_j(t)\sin(jx)\ ,\quad z^m(x,t)=\sum_{j=1}^m z_j(t)\sin(jx)
\endeq
and the approximated $m$-mode system
\begin{equation}\label{systemU}
\left \{ \begin{array}{ll}
\ddot{y}_j(t)+j^4 y_j(t)+\frac{4}{\pi}\int_0^\pi y^m(x,t)(1\!+\!y^m(x,t)^2\!+\!3z^m(x,t)^2)\sin(jx)\, dx=0\\
\ddot{z}_j(t)+j^2 z_j(t)+\frac{12}{\pi}\int_0^\pi z^m(x,t)(1\!+\!3y^m(x,t)^2\!+\!z^m(x,t)^2)\sin(jx)\, dx=0
\end{array}\right.
\end{equation}
where $j=1,...,m$. Then, suitable a priori estimates allow to prove that
$$y^m\rightarrow y \text{ in } C^0([0,T];H^2\cap H_0^1(0,\pi))\cap C^1([0,T];L^2(0,\pi))\quad \text{as }m\rightarrow +\infty\, ,$$
$$z^m\rightarrow z \text{ in } C^0([0,T];H_0^1(0,\pi))\cap C^1([0,T];L^2(0,\pi))\quad \text{as }m\rightarrow +\infty$$
for all $T>0$. Hence, the functions in \eq{fouriern} approximate the solutions of \eq{systemgamma}. The error committed when replacing $y$ with $y^m$
and $z$ with $z^m$ can be rigorously estimated, see \cite[Theorem 2]{bergaz}.\par
In what follows we focus our attention to the simplest case $m=1$. Then, system \eq{systemU} reads
\begin{equation}\label{unmodo}
\left\{\begin{array}{ll}
\ddot{y}_1+3y_1+\frac32 y_1^3+\frac92 y_1z_1^2=0\\
\ddot{z}_1+7z_1+\frac92 z_1^3+\frac{27}{2}z_1y_1^2=0\ ,
\end{array}\right.
\end{equation}
with some initial conditions
\neweq{initialone}
y_1(0)=\eta_0\, ,\ \dot{y}_1(0)=\eta_1\, ,\ z_1(0)=\zeta_0\, ,\ \dot{z}_1(0)=\zeta_1\ .
\endeq

If we take $\zeta_0=\zeta_1=0$, then the unique solution of \eq{unmodo}-\eq{initialone} is
$(y_1,z_1)=(\bar y,0)$ with $\bar y=\bar y(\eta_0,\eta_1)$ being the unique (periodic) solution of
\neweq{soloy}
\ddot{y}+3y+\frac32 y^3=0\ ,\qquad y(0)=\eta_0\, ,\ \dot{y}(0)=\eta_1\, .
\endeq
We call $\bar y$ the \emph{first vertical mode} with associated energy
\neweq{energyy}
E(\eta_0,\eta_1)=\frac{\dot{\bar{y}}^2}{2}+\frac32 \bar{y}^2+\frac38 \bar{y}^4\equiv\frac{\eta_1^2}{2}+\frac32 \eta_0^2+\frac38 \eta_0^4\, .
\endeq
Since we are interested in the stability of the solution $z_1\equiv0$ corresponding to $\zeta_0=\zeta_1=0$, we linearize system \eq{unmodo} around $(\bar y,0)$. The torsional component of the linearized system is the following Hill equation \cite{hill}:
\neweq{nonlinearhill}
\ddot{\xi}+a(t)\xi=0\quad\mbox{with}\quad a(t)=7+\frac{27}{2}\overline{y}(t)^2\,.
\endeq
We say that the first vertical mode $\bar y$ at energy $E(\eta_0,\eta_1)$ is {\bf torsionally stable} if the trivial solution of \eq{nonlinearhill}
is stable. By exploiting a stability criterion by Zhukovskii \cite{zhk}, in \cite{bergaz} we obtained the following theoretical estimates.
\begin{proposition}\label{stable}
The first vertical mode $\overline{y}$ at energy $E(\eta_0,\eta_1)$ (that is, the solution of \eqref{soloy})
is torsionally stable provided that
$$\|\overline{y}\|_\infty\le\sqrt{\frac{10}{21}}\approx0.69$$
or, equivalently, provided that
$$E\le\frac{235}{294}\approx0.799\, .$$
\end{proposition}
Proposition \ref{stable} gives a sufficient condition for the torsional stability. The numerical results obtained in \cite{bergaz} show that the threshold of
instability could be larger. We quote a couple of them in Figure 1. We plot the solution of \eq{unmodo} with initial conditions
\neweq{numinitial}
y_1(0)=\|y_1\|_\infty=10^4z_1(0)\, ,\quad\dot{y}_1(0)=\dot{z}_1(0)=0
\endeq
for different values of $\|y_1\|_\infty$. The green plot is $y_1$ and the black plot is $z_1$.
\begin{figure}[ht]
\begin{center}
{\includegraphics[height=26mm, width=52mm]{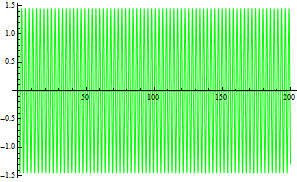}}\qquad\qquad {\includegraphics[height=26mm, width=52mm]{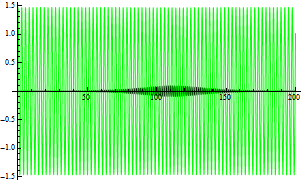}}

{\includegraphics[height=26mm, width=52mm]{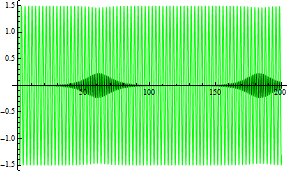}}\qquad\qquad {\includegraphics[height=26mm, width=52mm]{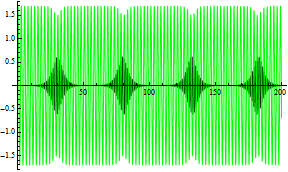}}
\caption{{\small For $t\in[0,200]$ (left to right, top to bottom): the solutions $y_1$ (green) and $z_1$ (black) of \eq{unmodo}-\eq{numinitial} for
$\|y_1\|_\infty=1.45,\, 1.47,\, 1.5,\, 1.7$.}}\label{tantefig}
\end{center}
\end{figure}
For $\|y_1\|_\infty=1.45$ no wide torsion appears, which means that the solution $(y_1,0)$ is torsionally stable. For $\|y_1\|_\infty=1.47$
we see a sudden
increase of the torsional oscillation around $t\approx50$. Therefore, the stability threshold for the vertical amplitude of oscillation lies in the
interval $[1.45,1.47]$. Finer experiments show that the threshold is $\|y_1\|_\infty\approx1.46$, corresponding to a critical energy of about $E\approx4.9$.
For increasing $\|y_1\|_\infty$ the phenomenon is anticipated in time and amplified in magnitude.

\section{How to introduce the aerodynamic forces into the model?}\label{aero}

Even in absence of wind, an aerodynamic force is exerted on the bridge by the surrounding air in which the structure is immersed, and is due to the
relative motion between the bridge and the air. Pugsley \cite[$\S$ 12.7]{pugsley} assumes that the aerodynamic forces depend linearly on the
``cross'' derivatives and functions. Similarly, Scanlan-Tomko \cite{scantom} obtain the following equations satisfied by the torsional angle $\theta$:
\neweq{scann}
I\, [\ddot{\theta}(t)+2\zeta_\theta\omega_\theta \dot{\theta}(t)+\omega_\theta^2\theta(t)]=A\dot{\theta}(t)+B\theta(t)\ ,
\endeq
where $I$, $\zeta_\theta$, $\omega_\theta$ are, respectively, associated inertia, damping ratio, and natural frequency. The r.h.s.\ of \eq{scann}
represent the aerodynamic force which is postulated to depend linearly on both $\dot{\theta}$ and $\theta$ with $A,B>0$ depending on the structural
parameters of the bridge. Let us mention that the arguments used in \cite{scantom} to reach the l.h.s.\ of \eq{scann} have been the object of severe
criticisms (see \cite{mckdcds}), due to some rough approximations and questionable arguments. Nevertheless, the r.h.s.\ of \eq{scann} is nowadays recognized
as a satisfactory description of aerodynamic forces.\par
Following these suggestions, we insert the aerodynamic forces in the 1-mode system \eq{unmodo}. We first consider the case where only
the cross-derivatives are involved. This leads to the following modified system:
\neweq{aerforce}
\left\{\begin{array}{ll}
\ddot{y}_1+3y_1+\frac32 y_1^3+\frac92 y_1z_1^2+\delta\dot{z}_1=0\\
\ddot{z}_1+7z_1+\frac92 z_1^3+\frac{27}{2}z_1y_1^2+\delta\dot{y}_1=0
\end{array}\right.
\end{equation}
with $\delta>0$. As in \eq{numinitial}, we take the initial conditions
\neweq{numinitial3}
y_1(0)=\sigma=10^4z_1(0)\, ,\quad \dot{y}_1(0)=\dot{z}_1(0)=0
\endeq
for different values of $\sigma$ and we wish to highlight the differences, if any, between \eq{unmodo} and \eq{aerforce}. For \eq{aerforce} we have no
energy conservation; however, let us consider the (variable) energy function
\neweq{variablenergy}
E(t)=\frac{\dot{y}_1^2}{2}+\frac{\dot{z}_1^2}{6}+\frac{3}{2}y_1^2+\frac{7}{6}z_1^2+\frac{9}{4}y_1^2z_1^2+\frac{3}{8}(y_1^4+z_1^4)\,.
\endeq

Let us now consider the case where also the cross-terms of order 0 are involved. Then, instead of \eq{aerforce} we obtain the system
\neweq{aerforce22}
\left\{\begin{array}{ll}
\ddot{y}_1+3y_1+\frac32 y_1^3+\frac92 y_1z_1^2+\delta(\dot{z}_1+z_1)=0\\
\ddot{z}_1+7z_1+\frac92 z_1^3+\frac{27}{2}z_1y_1^2+3\delta(\dot{y}_1+y_1)=0
\end{array}\right.
\end{equation}
where the coefficient 3 in the second equation comes from the variation of the energy
\neweq{altraenergia}
E(t)=\frac{\dot{y}_1^2}{2}+\frac{\dot{z}_1^2}{6}+\frac{3}{2}y_1^2+\frac{7}{6}z_1^2+\frac{9}{4}y_1^2z_1^2+\frac{3}{8}(y_1^4+z_1^4)+\delta y_1z_1\,.
\endeq
Also for \eq{aerforce22} we do not have energy conservation but the function $E$ in \eq{altraenergia} better approximates the internal energy.
It may be questionable whether to include the last term $\delta y_1z_1$ into $E$ since this term depends on the aerodynamic forces.
However, the behavior of $E$, which we analyze in the next section, does not depend on the presence of this term.

\section{Numerical results}\label{numres}

For \eq{aerforce} we first take $\sigma=1.47$ and we modify the aerodynamic parameter $\delta$.  To motivate this choice we note that no energy transfer seems to occur for $\sigma$ below this threshold, furthermore $\sigma=1.47$ is also the numerical threshold found when no aereodynamic force is inserted in the model, see
Section \ref{One}. In Figures \ref{aerforcefig} and \ref{aerforcefig10}
we plot both the behavior of the solutions (first line) and the behavior of the energy $E(t)$ (second line), for increasing values of $\delta$.
\begin{figure}[ht]
\begin{center}
{\includegraphics[height=26mm, width=52mm]{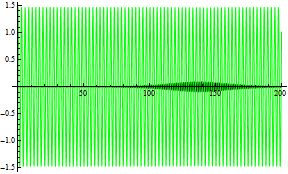}}\qquad\qquad {\includegraphics[height=26mm, width=52mm]{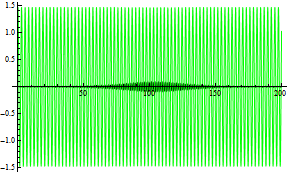}}

{\includegraphics[height=26mm, width=52mm]{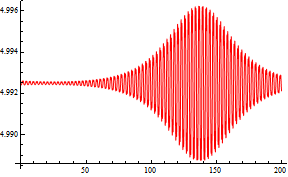}}\qquad\qquad {\includegraphics[height=26mm, width=52mm]{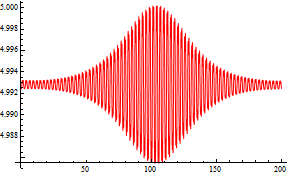}}
\caption{{\small For $t\in[0,200]$: the solutions $y_1$ (green) and $z_1$ (black) of \eq{aerforce}-\eq{numinitial3} for $\sigma=1.47$ and
$\delta=0.01,\, 0.02$ (left to right, first line). Second line (red): the energy $E=E(t)$ defined in \eq{variablenergy}.}}\label{aerforcefig}
\end{center}
\end{figure}

\begin{figure}[ht]
\begin{center}
{\includegraphics[height=26mm, width=52mm]{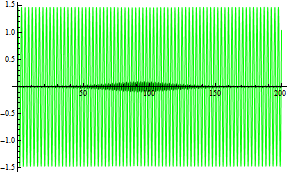}}\qquad\qquad {\includegraphics[height=26mm, width=52mm]{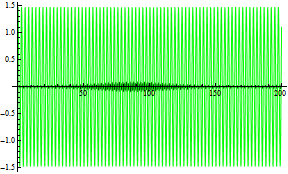}}

{\includegraphics[height=26mm, width=52mm]{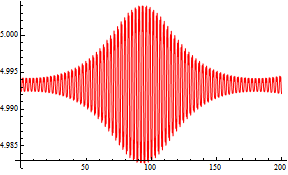}}\qquad\qquad {\includegraphics[height=26mm, width=52mm]{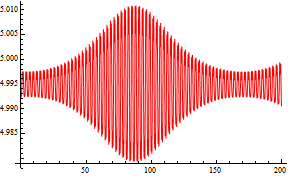}}
\caption{{\small For $t\in[0,200]$: the solutions $y_1$ (green) and $z_1$ (black) of \eq{aerforce}-\eq{numinitial3} for $\sigma=1.47$ and
$\delta=0.03,\, 0.05$ (left to right, first line). Second line (red): the energy $E=E(t)$ defined in \eq{variablenergy}.}}\label{aerforcefig10}
\end{center}
\end{figure}

The first lines in Figures \ref{aerforcefig} and \ref{aerforcefig10} should be compared with the second picture in Figure \ref{tantefig} (case $\delta=0$).
We note that, as the aerodynamic parameter increases, the transfer of energy is anticipated but it is not amplified. Quite surprisingly, on the second line we
see that the energy $E(t)$ remains almost constant except in the interval of time where the transfer of energy occurs: for increasing aerodynamic parameters
$\delta$ we observe increasing variations in the energy behavior.\par
Then we maintain fixed $\delta=0.01$ and we increase the initial energy, that is, the initial amplitude of oscillation. In Figures \ref{aerforcefig2} and
\ref{aerforcefig22} we plot both the behavior of the solutions (first line) and the behavior of the energy $E(t)$ (second line), for increasing values of $\sigma$.
\begin{figure}[ht]
\begin{center}
{\includegraphics[height=26mm, width=52mm]{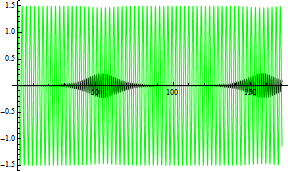}}\qquad\qquad {\includegraphics[height=26mm, width=52mm]{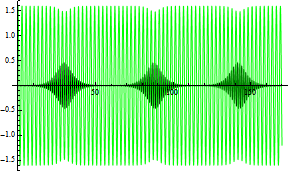}}

{\includegraphics[height=26mm, width=52mm]{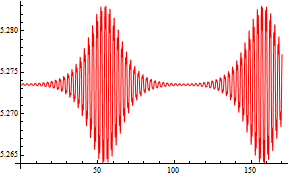}}\qquad\qquad {\includegraphics[height=26mm, width=52mm]{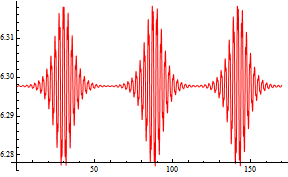}}
\caption{{\small For $t\in[0,170]$: the solutions $y_1$ (green) and $z_1$ (black) of \eq{aerforce}-\eq{numinitial3} for $\delta=0.01$
and $\sigma=1.5,\, 1.6$ (left to right, first line). Second line (red): the energy $E=E(t)$ defined in \eq{variablenergy}.}}\label{aerforcefig2}
\end{center}
\end{figure}

\begin{figure}[ht]
\begin{center}
{\includegraphics[height=26mm, width=52mm]{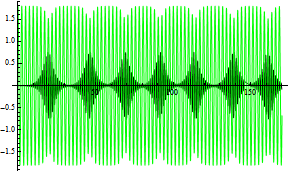}}\qquad\qquad {\includegraphics[height=26mm, width=52mm]{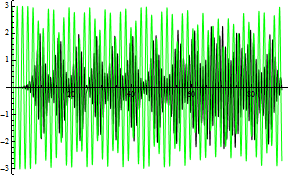}}

{\includegraphics[height=26mm, width=52mm]{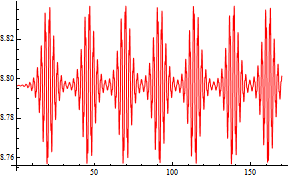}}\qquad\qquad {\includegraphics[height=26mm, width=52mm]{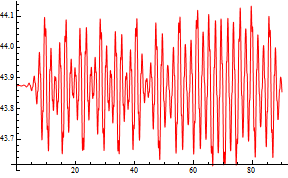}}
\caption{{\small For $t\in[0,170]$: the solutions $y_1$ (green) and $z_1$ (black) of \eq{aerforce}-\eq{numinitial3} for $\delta=0.01$
and $\sigma=1.8,\, 3$ (left to right, first line). Second line (red): the energy $E=E(t)$ defined in \eq{variablenergy}.}}\label{aerforcefig22}
\end{center}
\end{figure}
It turns out that all the phenomena are anticipated (in time) and amplified (in width) and reach a quite chaotic behavior for $\sigma=3$ where we had to stop
the numerical integration at $t=90$.\par
For \eq{aerforce22} we take again as initial conditions \eq{numinitial3} but with $\sigma\ge1.47$ so that we are above the energy threshold of instability,
see Section \ref{One}. In Figure \ref{newaerforce} we plot both the behavior
of the solution (first line) and the behavior of the energy $E(t)$ (second line) of \eq{aerforce22}-\eq{numinitial3}.
\begin{figure}[ht]
\begin{center}
{\includegraphics[height=26mm, width=52mm]{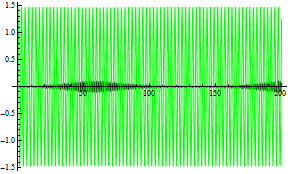}}\qquad\qquad {\includegraphics[height=26mm, width=52mm]{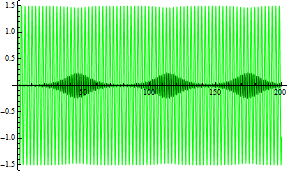}}

{\includegraphics[height=26mm, width=52mm]{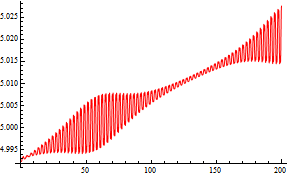}}\qquad\qquad {\includegraphics[height=26mm, width=52mm]{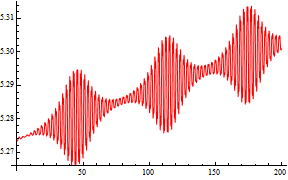}}
\caption{{\small For $t\in[0,200]$: the solutions $y_1$ (green) and $z_1$ (black) of \eq{aerforce22}-\eq{numinitial3} for $\delta=0.01$ and
$\sigma=1.47,\ 1.5$ (left to right, first line). Second line (red): the energy $E=E(t)$ defined in \eq{altraenergia}.}}\label{newaerforce}
\end{center}
\end{figure}

It is quite visible that the instability is further anticipated but now also the amplitude is enlarged. Moreover, the energy increases also
in absence of torsional instability: this variation is due to the cross-derivatives since all the other terms appear in the energy \eq{altraenergia}.
The very same behavior is obtained for the internal energy, namely the energy \eq{altraenergia} without the last term $\delta y_1z_1$. We also
remark that the energy $E$ fails to follow a regular pattern only in presence of instability. We only quote these numerical results because all
the other experiments gave completely similar responses.

\section{Theoretical results}
As pointed out by Irvine \cite[p.176]{irvine}, the numerical approach is probably the most appropriate to analyze a model which also involves
aerodynamic forces. The reason is that a satisfactory stability theory for systems such as \eq{aerforce} and \eq{aerforce22} is not available.
Nevertheless, some theoretical conclusions can be drawn also for these systems, in the spirit of the results obtained in \cite{bergaz} (see also
Section \ref{One}), where the main idea is to study the solutions of the systems near the pure vertical mode $\bar y$.
Here ``pure'' means that no interactions with the surrounding are admitted and only the structural behavior of the bridge is considered. Let us explain
how some of the results for the isolated system \eq{unmodo} may be extended to \eq{aerforce}; one can then proceed similarly for \eq{aerforce22}.\par
For system \eq{unmodo}, two steps are necessary to define the torsional stability of the unique (periodic) solution $\bar y$ of \eq{soloy}
(the pure mode):\par
(i) we linearize the torsional equation of the system \eq{unmodo} around $(\bar y,0)$, see \eq{nonlinearhill};\par
(ii) we say that the pure vertical mode $\bar y$ at energy $E(\eta_0,\eta_1)$ is torsionally stable if the trivial solution of \eq{nonlinearhill}
is stable.\par
We point out that the system \eq{unmodo} is {\em isolated} and that \eq{nonlinearhill} is {\em unforced}. In this situation, the above steps (i)-(ii) are
equivalent to:\par
(I) in the torsional equation of the system \eq{unmodo} we drop all the $z_1$-terms of order greater than or equal to one and we replace $y_1$
with $\bar y$;\par
(II) we say that the pure vertical mode $\bar y$ at energy $E(\eta_0,\eta_1)$ is torsionally stable if all the solutions of \eq{nonlinearhill} are
globally bounded.\par
If we replace \eq{unmodo} with the system \eq{aerforce}, then (i)-(ii) make no sense while (I)-(II) do. A linearization as in (i) would exclude the
aerodynamic forces, while acting as in (I) preserves them and gives rise to the forced Hill equation
\neweq{complete}
\ddot{\xi}+a(t)\xi=f(t)\quad\mbox{with}\quad a(t)=7+\frac{27}{2}\overline{y}(t)^2 \text { and }f(t)=-\delta \dot{\bar y}(t)\,,
\endeq
where $\bar y$ is the unique periodic solution of \eqref{soloy}. Needless to say, $f$ and $a$ have the same period. The definition (ii) is inapplicable
to \eq{complete} since $\xi\equiv0$ is not a solution, while (II) is a verifiable property for \eq{complete}. This is the definition of stability that
we adopt for the vertical mode $\bar y$ in presence of aerodynamic forces. With this definition we can prove the following statement.

\begin{proposition}\label{stable aero}
Let $\overline{y}$ be the pure vertical mode at energy $E(\eta_0,\eta_1)$, that is, the solution of \eqref{soloy}. Then $\bar y$ is torsionally stable
for \eqref{unmodo} if and only if it is torsionally stable for \eqref{aerforce}.
\end{proposition}
\begin{proof} Assume that $\bar y$ is torsionally stable for \eq{unmodo}, then the trivial solution of \eq{nonlinearhill} is stable (condition (i)) or,
equivalently, all the solutions of \eq{nonlinearhill} are globally bounded (condition (II)). Let $T$ be the period of $a(t)$, from the classical Floquet
theory any (stable) solution of \eq{nonlinearhill} may be written as
\neweq{xiAB}
\xi_h(t)=A\xi_1(t)+B\xi_2(t)=A e^{i \beta_1 t} \varphi_1(t)+B e^{i\beta_2 t} \varphi_2(t)\,,
\endeq
where $A,B\in \R$, $\beta_1 \neq \beta_2 $ are real numbers such that $e^{i \beta_j T}  \neq \pm 1$ ($j=1,2$) and $\varphi_1$ and $\varphi_2$
are $T$-periodic functions. By the variation of constants formula we then find that any solution of \eq{complete} takes the form
$$
\xi(t)=\xi_h(t)-\frac{1}{c^2}\Big[e^{i \beta_1 t} \varphi_1(t) \int_0^t e^{i \beta_2 s} \varphi_2(s)f(s)\,ds-e^{i \beta_2 t} \varphi_2(t)
\int_0^t e^{i \beta_1 s} \varphi_1(s)f(s)\,ds\Big]
$$
where $c^2$ is the (constant) Wronskian of $\xi_1$ and $\xi_2$, namely $c^2=\xi_1(t)\xi_2'(t)-\xi_2(t)\xi_1'(t)$. From \cite[Formula (53)]{bergaz} we
know that $|\dot{\bar y}(t)| \leq \sqrt{2E}$ for every $t>0$, hence $|f(t)| \leq \delta  \sqrt{2E}$ for every $t>0$. This proves the boundedness of the
general solution $\xi$ and, in turn, the torsional stability of $\bar y$ for \eq{aerforce}.\par
Assume now that $\bar y$ is torsionally stable for \eq{aerforce}. Then $\xi$ is bounded for any choice of $\xi_h$, that is, any choice of the constants
$A$ and $B$ in \eq{xiAB}. This shows that also $\xi_h$ is bounded and proves the stability of $\bar y$ for \eq{unmodo}. \end{proof}

This statement deserves a couple of straightforward comments.\par
$\bullet$ In agreement with the numerical results described in Section \ref{numres}, Proposition \ref{stable aero} shows that the energy threshold
for stability does not depend on the strength of the aerodynamic forces; in particular, an isolated system has the same energy threshold.\par
$\bullet$ By applying Propositions \ref{stable} and \ref{stable aero} we infer that the energy threshold for the stability of \eq{aerforce} is at least
$\frac{235}{294}$.\par

\section{Conclusions}\label{conclusion}

It is clear that in absence of wind or external sources a bridge remains still. A vertical load, such as a vehicle, bends the bridge and creates a
bending energy. Less obvious is the way the wind inserts energy into the bridge: let us outline how this happens. When a fluid hits a bluff body its
flow is modified and goes around the body. Behind the body, or a ``hidden part'' of the body, the flow creates vortices which are, in general,
asymmetric. This asymmetry generates a forcing lift which starts the vertical oscillations of the body. Up to some minor details, this explanation
is shared by the whole community and it has been studied with great precision in wind tunnel tests, see e.g.\ \cite{larsen,scott}.\par
The vortices induced by the wind increase the internal energy of the structure by generating wide vertical oscillations. When the amount of
energy reaches a critical threshold our results in \cite{bergaz} show that a structural instability appears: this is the onset of torsional oscillations.
The results in the present paper show that, at this stage, the aerodynamic forces excite the internal energy irregularly giving rise to further self-excited
oscillations.\par
The whole energy-oscillation mechanism is here described through a very simplified model which certainly needs to be significantly improved.
But, at least qualitatively, we believe that the ``true'' mechanism in a suspension bridge will follow this pattern:\par\noindent
1) the interaction of the wind with the structure creates vortices;\par\noindent
2) vortices create a lift which starts vertical oscillations of the bridge;\par\noindent
3) when vertical oscillations are sufficiently large, torsional oscillations may appear;\par\noindent
4) the onset of torsional instability is of structural nature;\par\noindent
5) the aerodynamic forces excite the energy only when the structural torsional instability appears;\par\noindent
6) the energy threshold of stability is independent of the strength of aerodynamic forces.

\section*{Acknowledgments}
\noindent
This work is supported by the Research Project FIR (Futuro in Ricerca) 2013 \emph{Geometrical and qualitative aspects of PDE's} and
by the PRIN project {\em Equazioni alle derivate parziali di tipo ellittico e parabolico: aspetti geometrici, disuguaglianze collegate, e applicazioni}.
The authors are members of the Gruppo Nazionale per l'Analisi Matematica, la Probabilit\`a e le loro Applicazioni (GNAMPA) of the Istituto Nazionale di Alta Matematica (INdAM).

\end{document}